\pdfoutput=1 
\documentclass[11pt]{amsart}

\usepackage{epsfig,overpic,comment}
\usepackage[usenames,dvipsnames,svgnames,table]{xcolor}
\usepackage[hyphens]{url}
\usepackage[pagebackref,linktocpage=true,colorlinks=true,linkcolor=Blue,citecolor=BrickRed,urlcolor=RoyalBlue]{hyperref}
\usepackage[msc-links,abbrev]{amsrefs}
\usepackage{amsmath,amsthm,amssymb,marginnote}
\usepackage{enumerate}
\usepackage[T1]{fontenc}

\newtheorem{theorem}{Theorem}[section]

\newtheorem{lemma}[theorem]{Lemma}

\theoremstyle{definition}
\newtheorem{note}[theorem]{Note}

\newcommand{\be}{\begin{equation}}
\newcommand{\ee}{\end{equation}}
\newcommand{\ol}{\overline}

\newcommand{\R}{\mathbf{R}}

\newcommand{\C}{\mathcal{C}}
\newcommand{\G}{\Gamma}
\newcommand{\M}{\mathcal{M}}
\newcommand{\dvol}{\textup{dvol}}

\renewcommand{\epsilon}{\varepsilon}
\renewcommand{\S}{\mathbf{S}}

\DeclareMathOperator{\dist}{dist}

\begin{document}
\setlength{\baselineskip}{1.2\baselineskip}

\title[Continuity of total curvatures of Riemannian hypersurfaces] 
{Continuity of total curvatures of \\Riemannian hypersurfaces}

\author{Mohammad Ghomi}
\address{School of Mathematics, Georgia Institute of Technology,
Atlanta, GA 30332}
\email{ghomi@math.gatech.edu}
\urladdr{www.math.gatech.edu/~ghomi}

\vspace*{-0.75in}
\begin{abstract}
We show that total generalized mean curvatures of hypersurfaces with positive reach in Riemannian manifolds, and convex bodies in Cartan-Hadamard spaces, are continuous with respect to Hausdorff distance. 
\end{abstract}

\date{\today \,(Last Typeset)}
\thanks{The research of the author was supported by NSF grant DMS-2202337.}

\maketitle


\section{Introduction}\label{sec:intro}
Let $\Gamma$ be a $\C^{1,1}$ hypersurface in a Riemannian $n$-manifold $M$, cooriented by a (continuous) unit normal vector field $\nu$. Then the principal curvatures 
$\kappa:=(\kappa_1,\dots,\kappa_{n-1})$ of $\Gamma$ with respect to $\nu$ are well-defined almost everywhere, by Rademacher's theorem, and the   \emph{total $r^{th}$  mean curvature} of $\Gamma$ is given by
$$
\M_{r}(\G):=\int_\G \sigma_r(\kappa), 
$$
where 
$\sigma_r(\kappa):=\sum_{1\leq i_1<\dots<i_r\leq n-1}\kappa_{i_1}\dots \kappa_{i_r}$ are the elementary symmetric polynomials, for $1\leq r\leq n-1$.
By convention, $\sigma_0:=1$, and $\sigma_r:=0$ for $r\geq n$.
 Up to  multiplicative constants, depending only on $n$, $\mathcal{M}_r(\Gamma)$ form the coefficients of the generalized Steiner's polynomial and Weyl's tube formula \cite{gray2004}. They are also known as quermassintegrals when $\Gamma$ is a convex hypersurface in Euclidean space $\R^n$ \cite{schneider2014}. Here we study the continuity of these fundamental objects. In particular, we show:

\begin{theorem}\label{thm:main}
Let $\Gamma$ be a closed cooriented hypersurface with positive reach embedded in a Riemannian manifold $M$. Suppose there exists a sequence of closed embedded hypersurfaces $\Gamma_i\subset M$ with uniformly positive reach, and coorientations consistent with that of $\Gamma$, such  that $\Gamma_i\to\Gamma$ with respect to Hausdorff distance. Then $\mathcal{M}_r(\Gamma_i)\to \mathcal{M}_r(\Gamma)$.
\end{theorem}

The \emph{reach} of $\Gamma$, denoted by $\text{reach}(\Gamma)$, is the supremum of $\epsilon\geq 0$ such that through each point of $\Gamma$ there pass a pair of (geodesic) balls of radius $\epsilon$ whose interiors are disjoint from $\Gamma$. If $\text{reach}(\Gamma)>0$, then $\Gamma$ is $\C^{1,1}$ \cite[Lem. 2.6]{ghomi-spruck2022}. Thus $\M_r(\Gamma)$ is well-defined.  If $\Gamma_i\to\Gamma$ with respect to Hausdorff distance, then $\Gamma_i\to\Gamma$ in $\C^{1}$-topology (Lemma \ref{lem:Gi}). In particular, if $\Gamma$ is cooriented by the unit normal vector field $\nu$, then we may choose unit normal vector fields $\nu_i$ along $\Gamma_i$ such that $\nu_i\to\nu$ in local coordinates. This is what we mean by the assumption that the coorientations of $\Gamma_i$ are \emph{consistent} with that of $\Gamma$.

Theorem \ref{thm:main} was established in $\R^n$ by Federer \cite[Thm. 5.9]{federer1959}. The general version should follow from the theory of smooth valuations \cite{alesker2007},   and convergence of normal cycles \cite{fu1994,zahle1986}. The latter can be reduced to the Euclidean case \cite[Thm. 4.13]{federer1959} via local charts, since positive reach is preserved under diffeomorphisms \cite{bangert1982}. Here we give a more direct and fairly self-contained argument via universal differential forms introduced by Chern \cite{chern1945}.

The prime motivation for this work is the next observation, which we again establish geometrically using Theorem  \ref{thm:main} and some recent results
for total curvatures \cite{ghomi-spruck2023total}. A \emph{Cartan-Hadamard manifold} $M$ is a  complete, simply connected manifold with nonpositive curvature. A subset of $M$ is \emph{convex} if it contains the geodesic connecting every pair of its points. A \emph{convex hypersurface} $\Gamma\subset M$ is the boundary of a compact convex set with interior points, which we assume to be cooriented by the outward normal. We define 
\begin{equation}\label{eq:def}
\mathcal{M}_r(\Gamma):=\lim_{\epsilon\searrow 0}\M_r(\Gamma^\epsilon),
\end{equation}
 where $\Gamma^\epsilon$ denotes the outer parallel hypersurface of $\Gamma$ at distance $\epsilon$. Note that $\mathcal{M}_r(\Gamma^\epsilon)$ is well-defined since $\text{reach}(\Gamma^\epsilon)\geq\epsilon$ and thus $\Gamma^\epsilon$ is $\C^{1,1}$ for $\epsilon>0$ \cite[Lem. 2.6]{ghomi-spruck2022}. Furthermore,
the limit exists since 
\begin{equation}\label{eq:nondecreasing}
\epsilon\mapsto\mathcal{M}_r(\Gamma^\epsilon)  \;\,\text{is nondecreasing},
\end{equation}
by \cite[Cor. 4.4]{ghomi-spruck2023total}, and $\mathcal{M}_r(\Gamma^\epsilon)\geq 0$ since $\Gamma^\epsilon$ is convex (and cooriented by the outward normal). See \cite{ghomi-spruck2022} for basic facts about convex sets and their distance functions in Cartan-Hadamard manifolds.

\begin{theorem}\label{thm:main2}
The total curvature functionals $\mathcal{M}_r$ are continuous on the space of convex hypersurfaces in a Cartan-Hadamard manifold with respect to Hausdorff distance.
\end{theorem}

This result simplifies a number of arguments, e.g., see \cite[Note 3.7]{ghomi-spruck2022} and \cite[Lem. 3.3]{ghomi-spruck2023minkowski},  related to the \emph{Cartan-Hadamard conjecture} on the isoperimetric inequality in spaces of nonpositive curvature. The conjecture follows  if the \emph{total Gauss-Kronecker curvature}
\begin{equation}\label{eq:conjecture}
\M_{n-1}(\Gamma)\geq|\S^{n-1}|
\end{equation}
for convex hypersurfaces $\Gamma$ in Cartan-Hadamard manifolds, where $|\S^{n-1}|$ is the volume of the unit sphere in $\R^n$  \cite{ghomi-spruck2022}. Our proof of Theorem \ref{thm:main2} employs an estimate for total curvatures of parallel hypersurfaces (Lemma \ref{lem:comparison}) which may be of further interest.

\section{Universal Differential Forms}
Let $M$ be a Riemannian manifold, and $\Gamma\subset M$ be a hypersurface cooriented  by a unit normal vector field $\nu$.
Let $T_pM$ be the tangent space of $M$ at a point $p$, $S_p\subset T_p M$ be the set of unit vectors, and $SM:=\{(p,u)\mid p\in M, u\in S_p\}$ denote the unit tangent bundle of $M$. Let  $\ol \nu\colon \Gamma\to SM$ be given by $\ol \nu(p):=(p,\nu(p))$. The following fact is established in \cite[Prop. 3.8]{bernig-brocker2003}. 

\begin{lemma}[Bernig-Br\"{o}cker \cite{bernig-brocker2003}]\label{lem:universal}
For $0\leq r\leq n-1$, there exists an $(n-1)$-form $\Phi_r$ on $SM$ such that 
$$
\M_r(\Gamma)=\int_\Gamma\ol\nu^*(\Phi_r),
$$
 for any $\C^{1,1}$ hypersurface $\Gamma\subset M$.
\end{lemma}

The forms $\Phi_r$ are called \emph{universal} \cite{bernig-brocker2003} because they do not depend on $\Gamma$.
The form $\Phi_{n-1}$ corresponds to $\Phi_0$ in Chern  \cite[p. 675]{chern1945}, and is also described by Borbely \cite{borbely2002}.
See \cite{ghomi2024} for a concise construction of $\Phi_r$  in terms of the connection forms of $M$ and dual one forms of the principal frame of $\Gamma$. 

We describe a geometric construction for $\Phi:=\Phi_{n-1}$, which avoids exterior algebra. This is of special interest in connection with conjecture \eqref{eq:conjecture}. Let $GK:=\sigma_{n-1}(\kappa)$ be the \emph{Gauss-Kronecker curvature} of $\Gamma$.
To motivate our approach,  note that in $\R^n$
$$
\nu^*(\dvol_{\S^{n-1}})=GK \dvol_\Gamma,
$$
where $\dvol$ stands for volume form.
So the Gauss-Kronecker curvature of $\Gamma$ is the Jacobian of $\nu$ viewed as the Gauss map $\Gamma\to\S^{n-1}$. 
Hence we may set
$
\Phi:=\dvol_{\S^{n-1}}.
$
 To extend this concept to $M$, note that each tangent space of $SM$ admits a  decomposition \cite[Sec. 1.3]{paternain1999} into ``horizontal'' and ``vertical'' components given by
 $$
 T_{(p,v)} SM=T_p M\oplus (v^\perp)\simeq T_p M\oplus T_v S_p,
 $$
 where $v^\perp\subset T_p M$ is the subspace orthogonal to $v$, which may be identified with $T_v S_p$ by parallel transport within $T_p M$. Let $\pi\colon T_{(p,v)}SM\to T_v S_p$ be  projection onto the vertical component.
  Then we set
  $$
  \Phi_{(p,v)}:=\pi^*(\dvol_{S_p})_{v}.
  $$
 To check this construction, let $e_i\in T_p\Gamma$ be an orthonormal set of principal directions with corresponding curvatures $\kappa_i$. 
 Note that $d\pi=\pi$ and $d\ol\nu(e_i)=(e_i, \nabla_{e_i}\nu)$, where $\nabla$ is the covariant derivative. Thus $d(\pi\circ\ol\nu)(e_i)= \pi(e_i,\nabla_{e_i} \nu)=\nabla_{e_i}\nu=\kappa_i e_i$. So we have
 \begin{multline*}
(\ol\nu^*\Phi)_p(e_1,\dots, e_{n-1})
=
(\pi\circ\ol\nu)^*(\dvol_{S_p})_{\nu(p)}(e_1,\dots, e_{n-1})\\
=
(\dvol_{S_p})_{\nu(p)}(\kappa_1e_1,\dots, \kappa_{n-1}e_{n-1})
=
GK \,(\dvol_\Gamma)_p(e_1,\dots, e_{n-1}).
\end{multline*}
Hence $\ol\nu^*\Phi=GK \,\dvol_\Gamma$, as desired.

\section{Proof of Theorem \ref{thm:main}}
Let $\nu$ be the  unit normal vector field coorienting $\Gamma$, and $u\colon M\to\R$ be the signed distance function of $\Gamma$ with respect to $\nu$.  For $0<\delta<\text{reach}(\Gamma)$, let $U:=u^{-1}([-\delta,\delta])$ be the \emph{tubular neighborhood} of $\Gamma$ of radius $\delta$. Then $u\in\C^{1,1}(U)$ \cite[Lem. 2.6]{ghomi-spruck2022}, which means that $u$ is $\C^{1,1}$ in a collection of local coordinate charts of $M$ covering $U$. Fix $\delta<\min\{\text{reach}(\Gamma)/2$, $\text{reach}(\Gamma_i)/2\}$. We assume $i$ is so large that $\Gamma_i\subset U$.

First assume that $\Gamma_i$ do not intersect $\Gamma$. Let $\Omega_i$ be the compact region between $\Gamma$ and $\Gamma_i$ in $U$. Choose a unit normal vector field $\nu_i$ along $\Gamma_i$ so that $\nu_i$ points into (away from) $\Omega_i$, if $\nu$ points away from (into) $\Omega_i$. Let $u_i$ be the signed distance function of $\Gamma_i$ with respect to $\nu_i$. Note that $u_i\in\C^{1,1}(U)$, since $\delta<\text{reach}(\Gamma_i)/2$. 

\begin{lemma}\label{lem:Gi}
$u_i\to u$ in $\C^1(U)$,  and $\nabla u_i$ are uniformly Lipschitz on $U$.
\end{lemma}
\begin{proof}
By \cite[Prop. 2.8]{ghomi-spruck2022}, the Hessians of $u$ and $u_i$ are uniformly bounded almost everywhere on $U$. Hence, on $U$, the gradients $\nabla u$ and $\nabla u_i$ are uniformly Lipschitz. Next we show that
$u_i\to u$ in $\C^1(U)$. 
For any point $p\in U\setminus \Gamma$ there exists a (geodesic) sphere $S\subset U$ of radius $|u(p)|>0$ centered at $p$ with $S\cap\Gamma=\{\ol p\}$. Similarly, assuming $i$ is so large that $p\not\in\Gamma_i$, there exists a sphere $S_i\subset U$ of radius $|u_i(p)|>0$ centered at $p$ with $S_i\cap\Gamma_i=\{\ol p_i\}$. Since $\Gamma_i\to\Gamma$ in Hausdorff distance, $u_i\to u$ in $\C^0(U)$. Thus
$$
\dist(\ol p_i,S)\leq\dist(S_i,S)\to 0,\quad\quad\text{and}\quad\quad
\dist(\ol p_i,\Gamma)\leq \dist(\Gamma_i,\Gamma)\to 0,
$$
where $\dist$ is the Riemannian distance in $M$.
Thus any limit point of $\ol p_i$ lies in $\Gamma\cap S=\{\ol p\}$, or $\ol p_i\to \ol p$.  It follows that $\nabla u_i(p)\to \nabla u(p)$, since
these gradients are unit tangent vectors at $p$ to geodesic segments $p\ol p$ and $p \ol p_i$. Since the gradients are uniformly Lipschitz, and $U\setminus\Gamma$ is dense in $U$, $\nabla u_i(p)\to\nabla u(p)$ for all $p\in U$. Finally, since $U$ has compact closure, and $\nabla u_i$ are uniformly Lipschitz, it follows that  $\nabla u_i\to\nabla u$ uniformly on $U$, which completes the proof.
\end{proof}

As pointed out in the introduction, the above lemma implies that  $\nu_i\to\nu$ in local coordinates (so the coorientations of $\Gamma_i$ are consistent with that of $\Gamma$). The above lemma together with Kirszbraun's extension theorem leads to:

\begin{lemma}\label{lem:extension}
There exists $N>0$ such that, for $i>N$, $\nu_i$ extends to a uniformly Lipschitz unit vector field on $\Omega_i$ which coincides with $\nu$ on $\Gamma$.
\end{lemma}
\begin{proof}
Suppose first that $U$ is parallelizable, so that there exists a smooth orthonormal frame field $e_j$ on $U$. Then any unit vector in $T_p U$ is identified with a point of $\S^{n-1}$ by 
$$
v\mapsto\big(\langle v,e_1\rangle ,\dots, \langle v,e_n\rangle\big),
$$
where $\langle\cdot,\cdot\rangle$ denotes the metric on $M$.
 In particular, the union of $\nu$ and $\nu_i$ yields a mapping $w_i\colon \partial\Omega_i\to\S^{n-1}$. By Lemma \ref{lem:Gi}, $w_i$ are uniformly Lipschitz, say with constant $L$. So by Kirszbraun's theorem, $w_i$ admits an $L$-Lipschitz extension $\Omega_i\to\R^n$, which we again denote by $w_i$. For any point $p\in\Omega_i$ let $\ol p\in\Gamma$ be  its nearest point. Then $\dist(p,\ol p)\leq\delta_i$, where $\delta_i$ denotes the maximum length in $\Omega_i$ of the geodesics orthogonal to $\Gamma$. Assume $N$ is so large that $\delta_i<1/L$. Then, by the triangle inequality,
$$
|w_i(p)|\geq |w_i(\ol p)|-|w_i(p)-w_i(\ol p)|\geq 1-L\delta_i>0.
$$
Hence $\nu_i:=w_i/|w_i|$ yields the desired extension. 

If $U$ is not parallelizable, cover $\Gamma$ by open topological balls $B_k\subset \Gamma$. Let $U_k\subset U$ be the cylindrical neighborhoods foliated by  geodesics in $U$ which are orthogonal to $B_k$. Then $U_k$ admits a smooth orthonormal frame field. So, as discussed above, there exists a unit normal vector field $\nu_{i}^k$ on $\Omega_i\cap U_k$ which is $L$-Lipschitz, and coincides with $\nu$ and $\nu_i$ on $\partial\Omega_i\cap U_k$. 
Let $\phi_k$ be a partition of unity on $U$ subordinate to $\{U_k\}$, and set 
$$
\nu_i:=\sum_k\phi_k \nu_i^k\Big/\Big |\sum_k\phi_k \nu_i^k\Big|.
$$
We claim that for $N$ sufficiently large, $\sum_k\phi_k \nu_i^k\neq 0$ and thus $\nu_i$ is well-defined.
Indeed for any point $p\in \Omega_i\cap U_k$, we have $\ol p\in \Gamma\cap U_k$. Thus
$$
|\nu_i^k(p)-\nu(\ol p)|=|\nu_i^k(p)-\nu_i^k(\ol p)|\leq L\delta_i.
$$
Consequently, if $p\in \Omega_i\cap U_k\cap U_\ell$, then 
$$
|\nu_i^k(p)-\nu_i^\ell(p)|\leq |\nu_i^k(p)-\nu(\ol p)|+|\nu(\ol p)-\nu_i^\ell(\ol p)|\leq2L\delta_i.
$$
So if $N$ is sufficiently large, $|\nu_i^k(p)-\nu_i^\ell(p)|\leq 1$; in particular, these vectors all lie in an open hemisphere. Hence  $\sum\phi_k \nu_i^k\neq 0$, which ensures that  $\nu_i$ is the desired extension.
\end{proof}

Let $\nu_i$ be the extension given by Lemma \ref{lem:extension}. Then 
by Lemma \ref{lem:universal} and Stokes theorem, there exists a constant $C$ independent of $i$ such that
\begin{equation}\label{eq:Mr}
|\M_r(\Gamma_i)-\M_r(\Gamma)|
=
\left|\int_{\partial \Omega_i} \ol \nu_i^*(\Phi_r)\right|
=
\left|\int_{\Omega_i} d(\ol \nu_i^*(\Phi_r))\right|
=
\left|\int_{\Omega_i} \ol \nu_i^*(d\Phi_r)\right|
\leq 
C|\Omega_i|,
\end{equation}
since $\nu_i$ are uniformly Lipschitz, and so the pullbacks $\ol \nu_i^*$ are uniformly bounded. Thus 
$|\M_r(\Gamma_i)-\M_r(\Gamma)|\to 0$, which concludes the proof in the case where $\Gamma_i\cap\Gamma=\emptyset$.

To prove the general case, let $\Gamma'\subset U$ be a hypersurface parallel to $\Gamma$, i.e., a level set of $u$ different from $\Gamma$.  Let  $\Omega'$ be the region between $\Gamma'$ and $\Gamma$. Then $\Gamma_i$ will be disjoint from $\Gamma'$ for $i$ sufficiently large. Let $\Omega_i'$ be the region between $\Gamma'$ and $\Gamma_i$. Since $\Gamma$, $\Gamma'$, and $\Gamma_i$ all have uniformly positive reach, the same argument for \eqref{eq:Mr} shows that
$$
|\M_r(\Gamma_i)-\M_r(\Gamma')|\leq C|\Omega_i'|, \quad\quad\text{and}\quad\quad |\M_r(\Gamma')-\M_r(\Gamma)|\leq C|\Omega'|,
$$
for some constant $C$. Thus, by the triangle inequality, 
$$
\lim_{i\to\infty}|\M_r(\Gamma_i)-\M_r(\Gamma)|
\leq
\lim_{i\to\infty} C(|\Omega_i'|+|\Omega'|)
\leq 
2C|\Omega'|.
$$
As $\Gamma'\to\Gamma$, we have $|\Omega'|\to 0$, which completes the proof.

\begin{note}\label{note:referee}
The above proof shows that Theorem \ref{thm:main} holds for any functional on the space of closed embedded hypersurfaces $\Gamma\subset M$ with bounded reach which is obtained by integrating a universal differential form as in Lemma \ref{lem:universal}. In particular, these functionals include boundary terms of Lipschitz-Killing curvatures; see Chern \cite{chern1945}, Cheeger–M\"{u}ller–Schrader \cite{cheeger1984}, and Fu–Wannerer \cite{fu2019}.
\end{note}

\section{Proof of Theorem \ref{thm:main2}}
Let $M$ be a Cartan-Hadamard manifold and 
$\Gamma$ be a convex hypersurface in $M$. Recall that $\Gamma^\epsilon$ denote the outer parallel hypersurfaces of $\Gamma$ at distance $\epsilon\geq 0$. Let $\Omega_\epsilon$ be the region between $\Gamma^\epsilon$ and $\Gamma$, and $K_M$ denote the sectional curvature of $M$.

\begin{lemma}\label{lem:comparison}
Let $C:=\sup_{\Omega_\epsilon}|K_M|$. Then
$$
|\M_r(\Gamma^\epsilon)-\M_r(\Gamma)|\leq \big((r+1)\M_{r+1}(\Gamma^\epsilon) + C\M_{r-1}(\Gamma^\epsilon)\big)\epsilon.
$$
\end{lemma}
\begin{proof}
Let $u$ be the distance function of $\Gamma$. Then $|\nabla u|=1$ on the exterior region of $\Gamma$, and \cite[Thm. 3.1]{ghomi-spruck2023total} quickly yields
$$
\M_r(\Gamma^\epsilon)-\M_r(\Gamma)\leq (r+1)\int_{\Omega_\epsilon}\sigma_{r+1}(\kappa^u)+C\int_{\Omega_\epsilon}\sigma_{r-1}(\kappa^u),
$$
where $\kappa^u=(\kappa_1^u,\dots,\kappa_{n-1}^u)$ refers to the principal curvatures of the level sets of $u$. More precisely, we apply 
\cite[Thm. 3.1]{ghomi-spruck2023total} to parallel hypersurfaces $\Gamma^\delta$ for $0<\delta<\epsilon$ and take the limit as $\delta\to 0$ to obtain the above inequality.
By the coarea formula
$$
\int_{\Omega_\epsilon}\sigma_{r+1}(\kappa^u)=\int_{0\leq t\leq\epsilon} \M_{r+1}(\Gamma^t)\leq \epsilon \M_{r+1}(\Gamma^\epsilon),
$$
where the last inequality is due to the monotonicity property \eqref{eq:nondecreasing}. Similarly, $\int_{\Omega_\epsilon}\sigma_{r-1}(\kappa^u)\leq \epsilon \M_{r-1}(\Gamma^\epsilon)$, which completes the proof.
\end{proof}

Suppose there exists a sequence of convex hypersurfaces 
$\Gamma_i\subset M$ such that $\Gamma_i\to\Gamma$ with respect to Hausdorff distance.  By the triangle inequality,
\begin{multline*}
|\M_r(\Gamma_i)-\M_r(\Gamma)|
\leq \\
|\M_r(\Gamma_i)-\M_r(\Gamma_i^\epsilon)|
+
|\M_r(\Gamma_i^\epsilon)-\M_r(\Gamma^\epsilon)|
+
|\M_r(\Gamma^\epsilon)-\M_r(\Gamma)|.
\end{multline*}
As $i\to \infty$, the middle term on the right hand side vanishes by Theorem \ref{thm:main}. To bound the first term, let $B\subset M$ be a closed ball which contains $\Gamma$ in its interior, and set $C:=\sup_B|K_M|$. For $i$ sufficiently large and small $\epsilon$, we have $\Gamma_i$, $\Gamma_i^\epsilon\subset B$; therefore, Lemma \ref{lem:comparison} yields
$$
|\M_r(\Gamma_i)-\M_r(\Gamma^\epsilon_i)|
\leq 
\big((r+1)\M_{r+1}(\Gamma^\epsilon_i) + C\M_{r-1}(\Gamma^\epsilon_i)\big)\epsilon.
$$
But $\M_{r+1}(\Gamma^\epsilon_i)\to \M_{r+1}(\Gamma^\epsilon)$ and $\M_{r-1}(\Gamma^\epsilon_i)\to \M_{r-1}(\Gamma^\epsilon)$  by Theorem \ref{thm:main}, since these hypersurfaces have uniformly positive reach. Thus 
$$
\lim_{i\to\infty} |\M_r(\Gamma_i)-\M_r(\Gamma)|
\leq
\big((r+1)\M_{r+1}(\Gamma^\epsilon) + C\M_{r-1}(\Gamma^\epsilon)\big)\epsilon+
|\M_r(\Gamma^\epsilon)-\M_r(\Gamma)|.
$$
Letting $\epsilon\to 0$ and recalling \eqref{eq:def} completes the proof.

\section*{Acknowledgments}
We thank Semyon Alesker, Andreas Bernig, Joe Fu, Mario Santilli, and Joel Spruck for helpful communications. Thanks also to the anonymous referees for corrections to earlier drafts of this work,  and suggesting the content of Note \ref{note:referee}.

\bibliography{references}

\end{document}